\documentclass[a4paper, 12pt, reqno, twoside]{article}

\usepackage[top=2cm, bottom=2cm, left=2cm, right=2cm]{geometry}

\usepackage{fancyhdr}

\usepackage[font=small,format=plain,labelfont=bf,up,textfont=it,up]{caption}

\usepackage[utf8]{inputenc}
\usepackage{amsmath, amsfonts, amssymb}
\usepackage{ mathrsfs } 
\usepackage{xfrac}
\usepackage{bbm} 
\usepackage{amsthm}
\usepackage{graphicx}
\usepackage{indentfirst}
\usepackage{multicol}
\usepackage{setspace}
\usepackage[x11names]{xcolor}
\usepackage{float}

\usepackage[pdftex,                %
implicit=true,
bookmarks=true,
breaklinks=true,
bookmarksnumbered = true,
colorlinks= true,
urlcolor= green,
anchorcolor = yellow,
citecolor=blue,
pdfborder= {0 0 0}
]{hyperref}%

\hypersetup{
colorlinks=true,
linkcolor=blue}

\newtheorem{thm}{Theorem}
\newtheorem{lem}[thm]{Lemma}

\newtheorem{prop}[thm]{Proposition}
\newtheorem{cor}[thm]{Corollary}
\newtheorem{defn}[thm]{Definition}

\newtheoremstyle{reptheorem}{}{}{\itshape}{}{\scshape}{}{ }{\thmname{#1}\mathrm{#3}}
\theoremstyle{reptheorem}

\newtheoremstyle{remark}{}{}{}{}{\scshape}{.}{ }{}
\theoremstyle{remark}

\newtheorem{rem}[thm]{Remark}
\newtheorem{rems}[thm]{Remarks}
\newtheorem{exm}[thm]{Example}

\title{Seifert circles, crossing number and the braid index of generalized knots and links}

\author{GUSTAVO CARDOSO\\
Departamento de Matem\'atica,\\
Universidade Federal de S\~ao Carlos,\\
CEP:~13565-905 - São Carlos - SP - Brazil.\\
e-mail:~\url{gustavo.cardoso2017@hotmail.com}\vspace*{4mm}\\
OSCAR~OCAMPO~\\
Departamento de Matem\'atica - Instituto de Matem\'atica e Estat\'istica,\\
Universidade Federal da Bahia,\\
CEP:~40170-110 - Salvador - BA - Brazil.\\
e-mail:~\url{oscaro@ufba.br}
}

\date{\today}

\begin{document}

\maketitle

\begin{abstract}

For classical links Ohyama proved an inequality involving the minimal crossing number and the braid index, then motivated from this Takeda showed an analogous inequality for virtual links. 
In this paper, we are interested in studying properties of links independent of the type of crossings, and for this reason, we introduce generalized crossings for diagrams and generalized Reidemeister-type moves. 
The aim of this work is to prove the same type of inequality mentioned above but now involving the total crossing number and the braid index of generalized knots and links. In particular, we show that the result holds for virtual singular links.

\end{abstract}

\let\thefootnote\relax

\pagestyle{fancy}
\fancyhead{} 
\fancyhead[RO]{\textbf{Seifert circles, crossing number and the braid index of generalized knots and links}}
\fancyhead[CE]{\textbf{G.~Cardoso and O.~Ocampo}}

\renewcommand{\headrulewidth}{0pt}



\section{Introduction}

In this paper, we are interested in generalizations of classical knots and links from the point of view of its diagrams. We would like to study geometric objects similar to knots, but such that their respective diagrams have non-classical crossings, for instance virtual, flat and singular crossings will be considered.  
From the end of the twentieth century, some generalizations of the notion of a classical link were introduced, in particular, we highlight the virtual knot theory (and related objects like welded or unrestricted links) introduced by Kauffman \cite{K} as well as the notion of generalized knots recently defined by Bartholomew and Fenn 
\cite{BF1}. However, we note that the approach in this work for generalized links is independent of the one considered in \cite{BF1}. 


The notion of Seifert graphs was first introduced in \cite{book2} for diagrams of classical links and was proved in detail in \cite{book1} that, in this case, the Seifert graph of a link diagram is bipartite and planar. 
It is mentioned in \cite{article1} that we have the same result for the Seifert graph of a virtual link diagram (and also for flat virtual, welded, unrestricted virtual and singular links diagrams). One of the goals of this work is to give a proof of the fact that we can get this result for all cases previously mentioned and for any diagram of a generalization of classical links. In other words, we shall prove that the Seifert graph of any generalized knot or link diagram is bipartite and planar. As far as we know, this result is new in the literature for diagrams of doodles, virtual doodles and virtual singular links.

There are some interesting inequalities in classical knot theory that involve some well known numbers like the braid index, the crossing number, among others.
For instance, the Morton-Franks-Williams inequality for a classical link gives a lower bound for the braid index in terms of the HOMFLY polynomial, see \cite{FW} and \cite{M}. 
For classical links, Ohyama proved (\cite[Theorem~3.8]{article1}) the following inequality
\begin{equation}\label{eq:ohyama}
c(L)\geq 2(b(L)-1),
\end{equation}
where $L$ denotes a nonsplit link and $c(L)$  and $b(L)$ are the minimal crossing number and the braid index  of $L$, respectively. 
Takeda extended this result for a virtual link $L$ under certain assumptions \cite[Theorem~3.1]{article3}, proving that
$
tc(L)\geq 2(vb(L)-1),
$
where $tc(L)$ is the total crossing number of $L$ and $vb(L)$ is the virtual braid index of $L$. 

The main aim of this paper is to prove the same type of inequality given in \eqref{eq:ohyama}, but for generalized knots and links involving the total crossing number and the braid index in this new context.
In order to do that, we will define generalized crossings for diagrams and generalized Reidemeister-type moves. 
The proof was inspired by the technique used by Ohyama in the classical case \cite{article1} and by Takeda in the virtual case \cite{article3}, in this case using the notion of generalized Seifert circles and the respective generalized Seifert graph.

This article is organized as follows. 
We start recalling in Section~\ref{sec2} the notion of a diagram for classical links and some of its generalizations, for instance links with virtual, flat or singular crossings. For each collection of diagrams, depending on its type of crossings, we mention the Reidemeister-type movements allowed in each collection. 
Then, we consider links with arbitrary crossings, and so we define for them the generalized braid index and also consider generalized Reidemeister-type moves.
Section~\ref{sec3} has two subsections. In the first one, we establish some basic facts of graph theory that are necessary to prove the main results of this section. 
In the second subsection, we construct the Seifert graph for generalized links and then we prove in Theorem~\ref{sggnl} that for a diagram of a generalized knot or link its Seifert graph is planar and bipartite. We finish the section stating in Theorem~\ref{thm1} that if $D$ is a diagram of a generalized link $L$ with no nugatory crossings, then 
$$
tc(D) \geq 2(S(D) - \textrm{ind}(D) - 1),
$$ 
where $tc(D)$ is the total crossing number of $D$, $S(D)$ is the number of Seifert circles of $D$ and $\textrm{ind}(D)$ is the index of the graph related to a diagram of an oriented generalized link (see Definition~\ref{def:ind}).
In Section~\ref{sec4} we prove the main result of the paper, see Theorem~\ref{thm2}, that is as follows.  
Let $D$ be a diagram of a generalized link $L$ under certain assumptions on its crossings, then we have the inequality 
$$ 
tc(L) \geq 2(gb(L) - 1), 
$$ 
where $tc(L)$ is the total crossing number of $L$ and $gb(L)$ is the generalized braid index of $L$. 
 We end the section showing that, in particular, the last inequality holds for virtual singular links, see Proposition~\ref{cor:vsl}.


\subsection*{Acknowledgments}

The majority of this work was completed within the scope of the undergraduate scientific initiation of the first author, year 2021/2022, under the guidance of the second author, as part of the research project ``Virtual braid groups'', registered in ``Sistema de Gerenciamento de Bolsas de Inicia\c c\~ao -- SISBIC/UFBA'' with number 21039. The first author was partially supported by PIBIC/UFBA - CNPq and by CAPES.


\section{Generalizations of knots and links}\label{sec2}

We define a knot as the image of the circle $\mathbb{S}^{1}$ in $\mathbb{R}^{3}$ under an embedding. A link is a collection of knots which do not intersect, in particular a knot is a link with one component. 
We refer the reader to the books \cite{book1}, \cite{book2}, and \cite{book3} for more details on knot theory.
In this work, we call these geometric objects classical links. Knots and links are presented usually by regular projections on the plane $\mathbb{R}^{2}$ which are called their \emph{diagrams}. In the classical case, we consider (real) crossings as illustrated in Figure~\ref{real crossings}. 

\begin{figure}[!htb]
\centering
\includegraphics[scale=0.15]{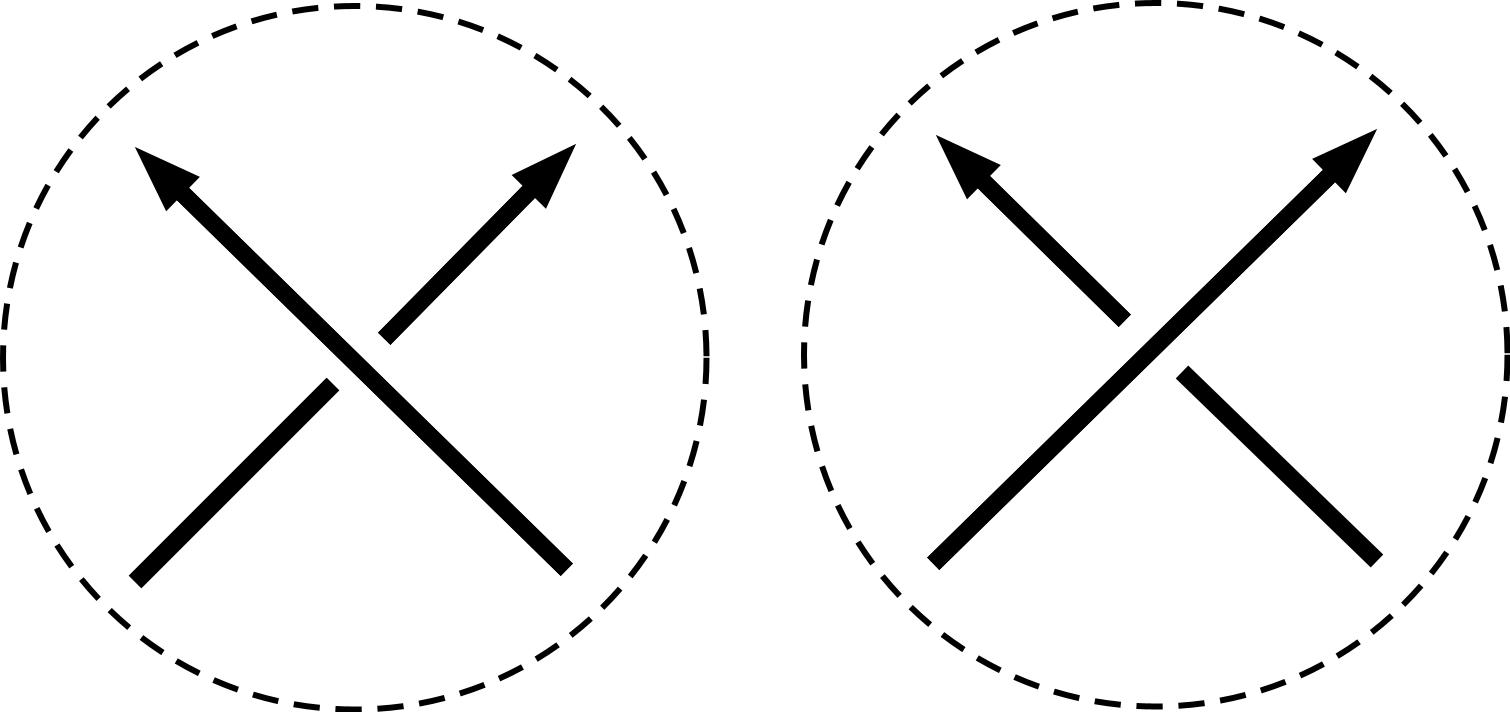}
\caption{Negative and positive real crossings, respectively.}
\label{real crossings}
\end{figure}

The Reidemeister's theorem \cite{R} asserts that two knots (or links) are equivalents if and only if their respective diagrams can be transformed into each other by a finite number of Reidemeister moves, represented in Figure~\ref{RM}.

\begin{figure}[!htb]
\centering
\includegraphics[scale=1.0]{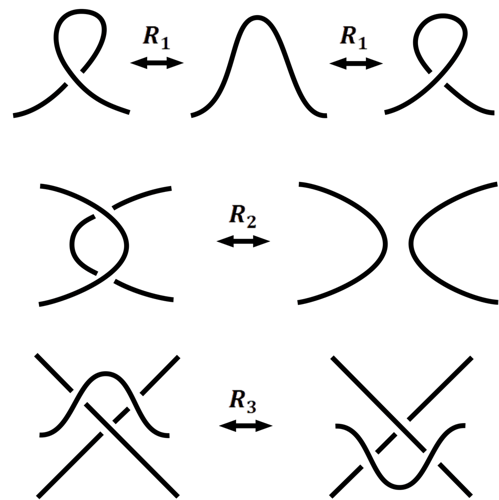}
\caption{Reidemeister moves.}
\label{RM}
\end{figure}

Every classical knot or link diagram can be regarded as an immersion of circles in the plane with extra structure at the double points. 
This extra structure is usually indicated by the over and under crossing conventions that give instructions for constructing an embedding of the link in three-dimensional space from the diagram. But we can consider this extra structure as other types of crossings. In this work we shall consider other kind of crossings, for instance the flat, virtual and singular crossings, illustrated in Figure~\ref{FVS}.

\begin{figure}[!htb]
\centering
\includegraphics[scale=0.15]{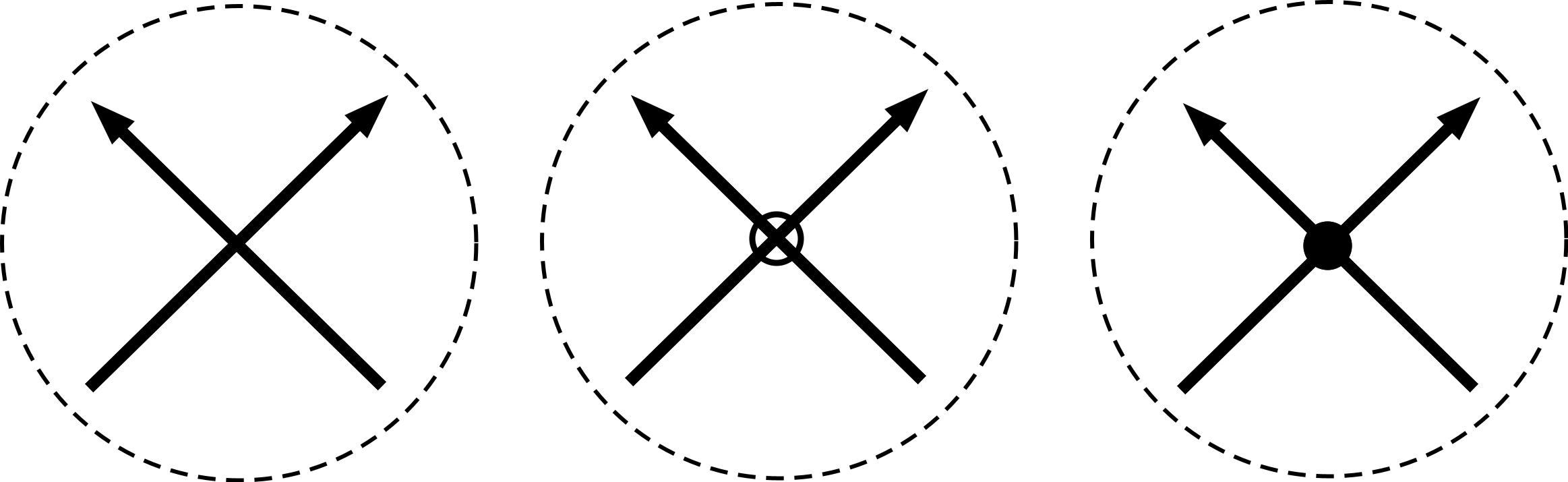}
\caption{Flat, virtual and singular crossings.}
\label{FVS}
\end{figure}

More precisely, a {\it virtual diagram} is a diagram with real and virtual crossings. A {\it flat virtual diagram} is a diagram with virtual and flat crossings. {\it Welded} and {\it unrestricted virtual diagrams} are diagrams with real and virtual crossings, differing by a forbidden move, as we shall see later. A {\it singular diagram} is a diagram with real and singular crossings and a {\it virtual singular diagram} is a diagram with real, virtual, and singular crossings.

As well as in classical knot theory  we have the same behavior for the generalized knot theories mentioned above by using the so-called {\it Reidemeister-type moves}, illustrated below that are analogous to the three Reidemeister moves of Figure~\ref{RM}.

\begin{figure}[!htb]
\centering
\includegraphics[scale=1.0]{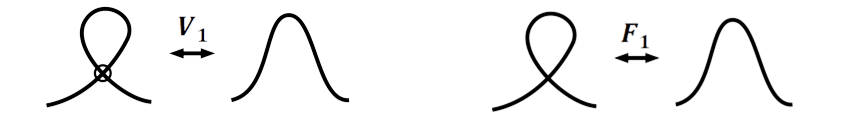}
\caption{$R_{1}$-type moves.}
\label{R1-type}
\end{figure}

\begin{figure}[!htb]
\centering
\includegraphics[scale=1.0]{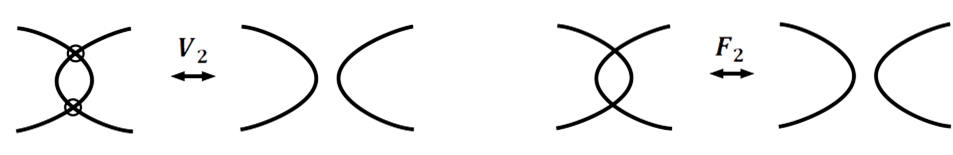}
\caption{$R_{2}$-type moves.}
\label{R2-type}
\end{figure}

\begin{figure}[!htb]
\centering
\includegraphics[scale=1.0]{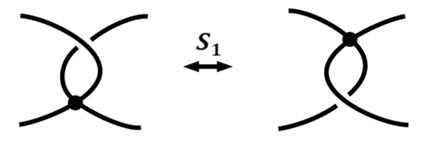}
\caption{Move allowed with real and singular crossings.}
\label{S1}
\end{figure}

\begin{rem}
We note that there no exists a Reidemeister move of type two for singular crossings, instead of it we have a move involving a real and a singular crossing, see Figure~\ref{S1}.
\end{rem}

\begin{figure}[!htb]
\centering
\includegraphics[scale=1.0]{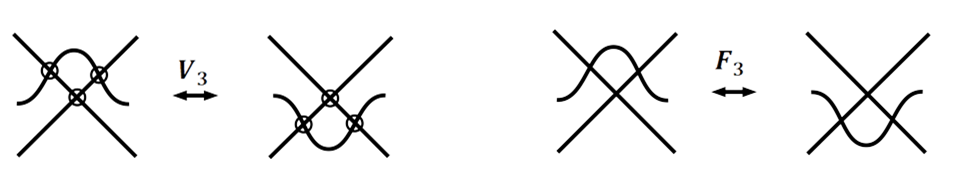}
\caption{$R_{3}$-type moves with the same crossings.}
\label{R3.1-type}
\end{figure}

\begin{figure}[!htb]
\centering
\includegraphics[scale=1.0]{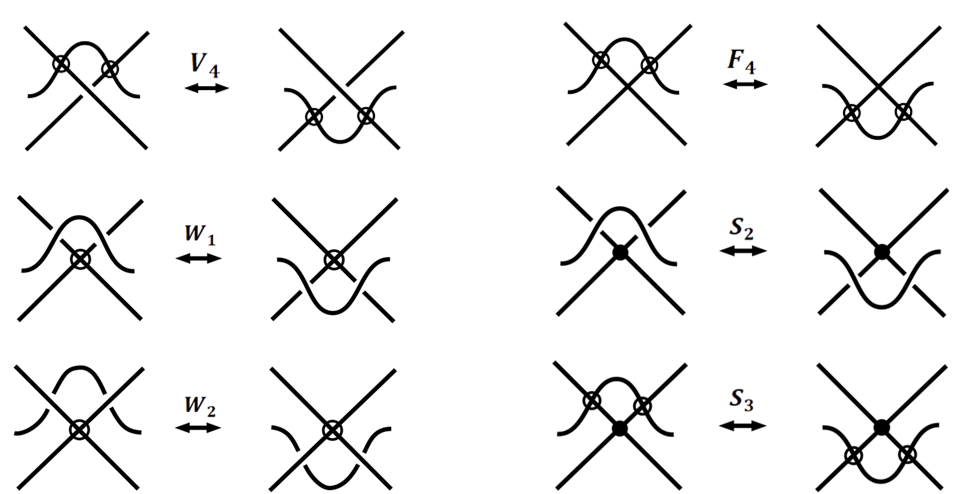}
\caption{$R_{3}$-type moves with distinct crossings.}
\label{R3.2-type}
\end{figure}

\begin{rems}
  \begin{itemize}
      \item[(i)] When considering more than one type of crossing, we have a Reidemeister-type move $R_3$ with the three crossings equal or with exactly two types of crossing; see Figures~\ref{R3.1-type} and \ref{R3.2-type};

      \item[(ii)] The so-called forbidden moves are illustrated in Figure~\ref{R3.2-type}, namely $W_1$ and $W_2$.
  \end{itemize}  
\end{rems}

\begin{table}[H]
\centering
    \begin{tabular}{c|c|c}
   {\it Diagrams}  &  {\it Crossings} & {\it Moves } \\
   \hline 
   Classical  & Real & $R_{1}, R_{2}, R_{3}$ \\
   \hline
   Virtual & Real and virtual & $R_{1}, R_{2}, R_{3}, V_{1}, V_{2}, V_{3}, V_{4}$\\
\hline
   Flat Virtual & Flat and virtual & $F_{1}, F_{2}, F_{3}, F_{4}, V_{1}, V_{2}, V_{3}$\\
\hline
   Welded & Real and virtual & $R_{1}, R_{2}, R_{3}, V_{1}, V_{2}, V_{3}, V_{4}, W_{1}$\\
\hline
   Unrestricted & Real and virtual & $R_{1}, R_{2}, R_{3}, V_{1}, V_{2}, V_{3}, V_{4}, W_{1}, W_{2}$ \\
\hline
   Singular & Real and singular & $R_{1}, R_{2}, R_{3}, S_{1}, S_{2}$\\
   \hline
   Virtual singular & Real, virtual  & $R_{1}, R_{2}, R_{3}, V_{1}, V_{2}, V_{3}, V_{4}, $\\ 
   & and singular & $S_{1}, S_{2}, S_{3}$ \\
\hline
   Doodle & Flat & $F_{1}, F_{2}$\\
\hline
   Virtual doodle & Flat and virtual & $F_{1}, F_{2}, F_{4}$\\
   \hline
\end{tabular}
\caption{Some collections of generalized diagrams.}
\label{table}
\end{table}

We summarize in Table~\ref{table} some collections of diagrams according to the types of crossings and the allowed moves for each, characterizing the equivalence classes in each family. We include here also other generalizations of (virtual) knots and links, the so-called {\it doodles} \cite{FT} and {\it virtual doodles} \cite{BFK}, that are planar versions of (virtual) knots.

For any virtual singular link $L$ there are four types of crossing numbers: the minimal number of real crossings, the minimal number of virtual crossings, the minimal number of singular crossings and the minimal total number of crossings. In this paper, we study the fourth one: the {\it total crossing number} of a virtual singular link.
On the other hand, we have the notion of a {\it virtual singular braid} and the {\it closure} of a virtual singular braid. See Figures~\ref{braid} and \ref{closure} for an example.

\begin{figure}[!htb]
\begin{minipage}[c]{7cm}
\centering
\centering
\includegraphics[scale=0.7]{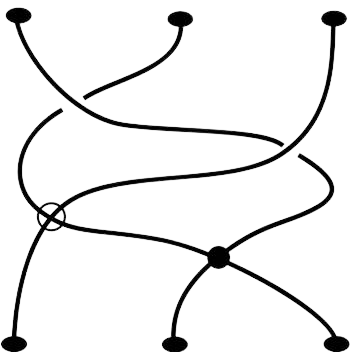}
\caption{A virtual singular braid.}
\label{braid}
\end{minipage}
\begin{minipage}[c]{9cm}
\centering
\centering
\includegraphics[scale=0.7]{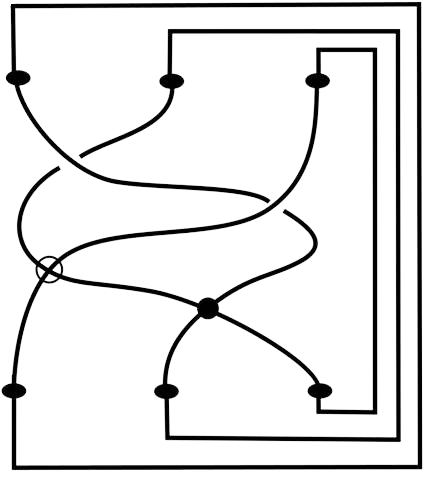}
\caption{Closure of a virtual singular braid.}
\label{closure}
\end{minipage}    
\end{figure}

Caprau, de la Pena and McGahan \cite{article2} showed that every virtual singular link is represented as the closure of a virtual singular braid. Therefore, we can define the {\it virtual singular braid index} of a virtual singular link as the minimal number of strings of a virtual singular braid whose closure is equivalent to the original virtual singular link.

In general, for any generalization of links for which we have an Alexander-type theorem (i.e. a generalization of the classical Alexander theorem \cite{A}), we may define the {\it generalized braid index} of a generalized link as the minimal number of strings of a generalized braid whose closure is equivalent to the original generalized link. In this case, we are also interested in the total crossing number of a generalized link. 
In \cite[Theorem~6.1]{BF1} the authors showed a generalized Alexander theorem for generalized knot theories under certain conditions, so, for them, the generalized braid index is well defined.

For the purposes of this work, we define the {\it generalized Reidemeister-type moves} according to Figure~\ref{GRM}. 
Note that for a $GR_{4}$ move it is necessary at least two distinct crossings.

\begin{figure}[!htb]
\centering
\includegraphics[scale=1.0]{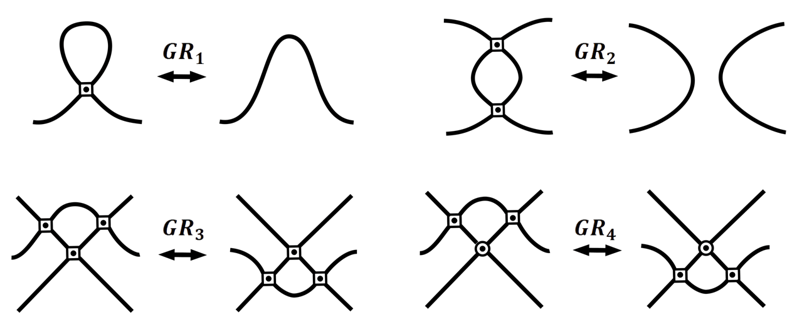}
\caption{Generalized Reidemeister-type moves.}
\label{GRM}
\end{figure}


\section{Seifert circles and the total crossing number of generalized links}\label{sec3}

The aim of this section is to prove that the Seifert graph of a generalized link diagram is planar and bipartite (see Theorem~\ref{sggnl}). 
To reach our goal, we first need some basic definitions and facts about graph theory.

\subsection{Graph theory}

In this subsection we record some relevant definitions and results that we shall use in this paper. It was extracted from \cite{article3}.
A {\it graph} $G$ is frequently the geometric realization of a combinatorial graph as a finite 1-dimensional $CW$-complex in $\mathbb{R}^{3}$. A graph $G$ is said to be {\it signed} if $+1, -1$ or $0$, called a {\it sign}, is assigned to each edge. If an edge is assigned by $+1$, the signed edge is said to be {\it positive}. If an edge is assigned by $-1$, the signed edge is said to be {\it negative}. 
If an edge is assigned by $0$, the signed edge is said to be {\it virtual}. A graph $G$ is said to be {\it bipartite} if any cycle has an even length and is said to be {\it planar} if $G$ is a graph embedded in $\mathbb{R}^{2}$. 
A graph $G$ is said to be {\it separable} if there are two subgraphs $J$ and $K$ such that $G = J \cup K$ and $J \cap K = \{v_{0}\}$, where $J$ and $K$ both have at least one edge and $v_{0}$ is a vertex, which is called a {\it cut vertex}. Otherwise, $G$ is said to be {\it non-separable}. A {\it block} is a maximal non-separable connected subgraph of $G$. A connected graph is decomposed into finitely many blocks and if $G_{1}, G_{2}, \ldots, G_{k}$ are the blocks, we write $G = G_{1} \ast G_{2} \ast \ldots \ast G_{k}$ and we say that $G$ is {\it the block sum} of $G_{1}, G_{2}, \ldots, G_{k}$. 

If two or more edges have common end points, then these edges are called {\it multiple edges}, and if two vertices are jointed by exactly one edge $e$, then $e$ is called a {\it singular edge} of $G$. A {\it cut edge} of $G$ is an edge whose removal increases the number of connected components. For a vertex $v$, $\mbox{star}(v)$ denotes the smallest subgraph of $G$, then $G/X$ is defined as the graph obtained from $G$ by identifying all points in $X$ to one point.

\begin{defn}
    Let G be a graph. A set $F={e_{1}, e_{2}, \ldots, e_{k}}$ of edges of $G$ is said to be {\it independent} if:
    \begin{enumerate}
        \item All $e_{j}$, for $1 \leq j \leq k$, are singular;

        \item No two of them are adjacent;

        \item There exists an edge $e_{i}$ in $F$ and a vertex $v$, one of the end points of $e_{i}$, such that $e_{1}, e_{2}, \ldots, e_{i-1}, \ldots, e_{k}$ is an independent set of $k-1$ edges in the graph $G/\mbox{star}(v)$.
    \end{enumerate}
    We assume that the empty set of edges is independent.
\end{defn}

We define $\textrm{ind}(G)$ as the maximal number of independent edges in $G$. If $G$ is a signed graph, then $\textrm{ind}_{0}(G)$ is defined to be the maximal number of independent edges in $G$, where all edges are singular and virtual. Similarly, $\textrm{ind}_{+}(G)$, $\textrm{ind}_{-}(G)$, $\textrm{ind}_{0-}(G)$, $\textrm{ind}_{+-}(G)$ are also defined. 
For example, $\textrm{ind}_{+}(G)$ is the maximal number of independent edges in $G$, where all edges are singular and positive.

The following interesting result about the index of a bipartite graph will be very useful.

\begin{thm}[{\cite[Theorem~2.4]{book2}}]\label{thmind}
    Let $G$ be a bipartite graph. If $G$ consists of blocks $G_{1}, G_{2}, \ldots, G_{k}$, then we have $$ind(G) = \sum^{k}_{i=1} ind(G_{i}).$$
\end{thm}

One consequence of the above theorem is the following.

\begin{cor}[{\cite[Corollary~2.3]{article3}}]\label{cor:ind}
    Let $G$ be a bipartite graph. If $G$ consists of blocks $G_{1}, G_{2}, \ldots, G_{k}$, then we have $$ind_{0}(G) = \sum^{k}_{i=1} ind_{0}(G_{i}).$$
\end{cor}

 By $|E(G)|$ and $|V(G)|$ we denote the numbers of edges and vertices in $G$, respectively. 
To finish this subsection, we state the next result.

 \begin{thm}[{\cite[Theorem~3.7]{article1}}]\label{2.4}
     If $G$ is a connected, planar, bipartite graph without a cut edge, then we have $$|E(G)| \geq 2(|V(G)| - ind(G) - 1)$$
 \end{thm}

\subsection{The Seifert graph of generalized knots and links diagrams}

{\it Generalized Seifert circles} are circles obtained by smoothing all crossings of a generalized link diagram as illustrated in Figure~\ref{Smoothing}.

\begin{figure}[!htb]
\centering
\includegraphics[scale=0.13]{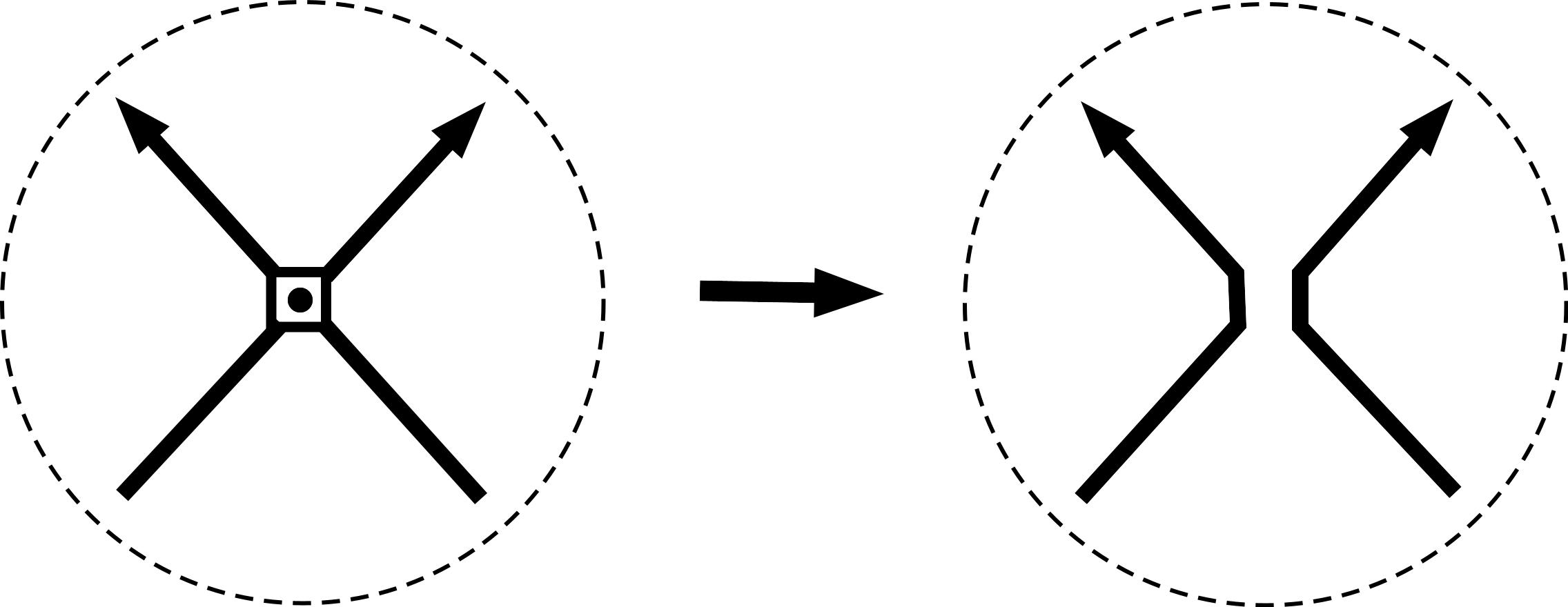}
\caption{A smoothing for a generalized crossing.}
\label{Smoothing}
\end{figure}

\begin{exm}\label{exm1}
After smoothing the crossings in the virtual  singular knot illustrated in Figure~\ref{Seifert circles}, we obtain two Seifert circles.
 
 \begin{figure}[!htb]
\centering
\includegraphics[scale=0.15]{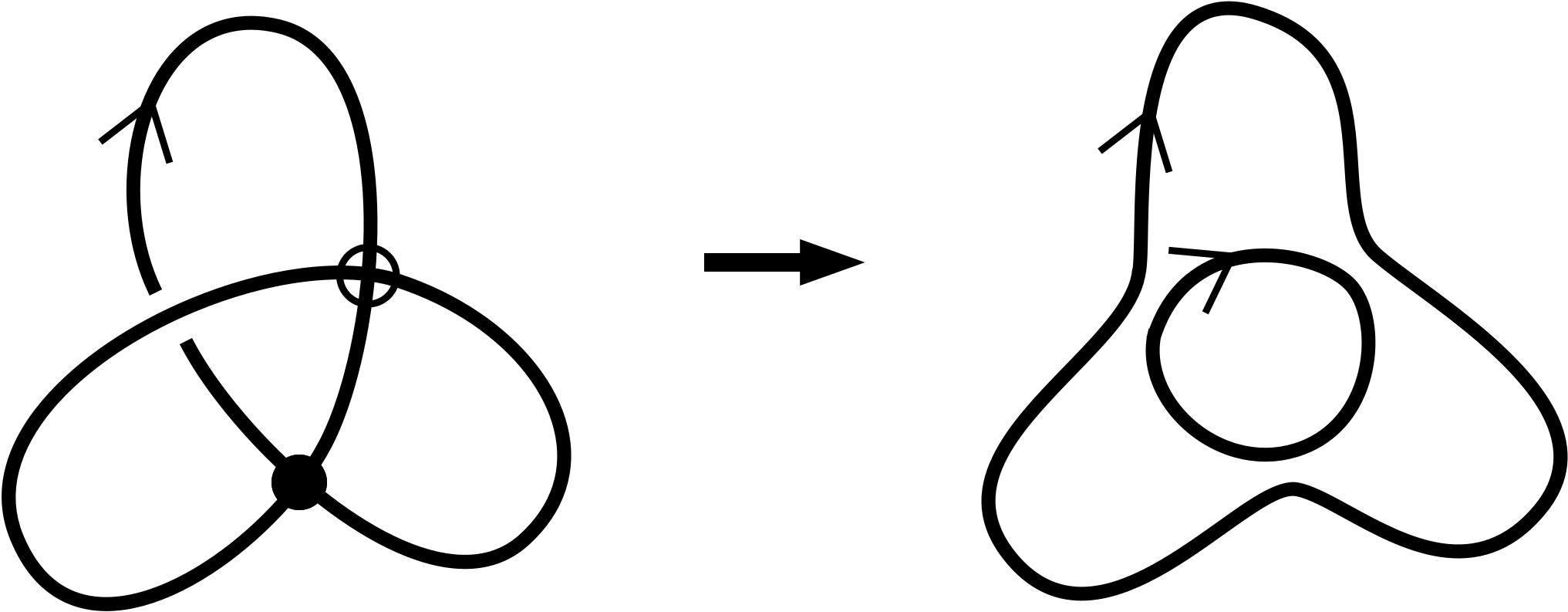}
\caption{The Seifert circles of a virtual singular knot.}
\label{Seifert circles}
\end{figure}
\end{exm} 

The following definitions were motivated by the ones related to virtual crossings, see \cite[Section~2]{article3}.
Let $L$ be a generalized link and $D$ its diagram. If $C$ is a Seifert circle, then it decomposes the plane into two closed regions $U$ and $V$ meeting along $C$. We say that $C$ is {\it separating} if both $(U \setminus C) \cap D$ and $(V \setminus C) \cap D$ are non-empty. Otherwise, $C$ is {\it non-separating}. If $C$ is separating, then let $D_{1}$ and $D_{2}$ be the diagrams constructed from $D \cap U$ and $D \cap V$, respectively, by filling the gaps with arcs from $C$ when they are necessary. 
We say that $D$ is a $\ast$-product of $D_{1}$ and $D_{2}$, and write $D=D_{1} \ast D_{2}$. 
A diagram $D$ is {\it special} if it does not decompose as a $\ast$-product: in other words, $D$ is special if and only if it has no separating Seifert circle. A general oriented diagram $D$ can be decomposed along its separating Seifert circles into a product $D_{1} \ast D_{2} \ldots \ast D_{r}$ of special diagrams. A crossing of $D$ is said to be {\it nugatory} if the number of connected components of a smoothed diagram at the crossing is greater than that of $D$.

The notion of Seifert graphs was first introduced in \cite{book2} for classical links. 
For a given diagram $D$ of a generalized link $L$, let $S(D)$ be the number of Seifert circles of $D$ and $c(D)$ the number of crossings in $D$. The {\it Seifert graph} $\Gamma (D)$ is a graph with $S(D)$ vertices $v_{1}, v_{2}, \ldots, v_{S(D)}$ and $c(D)$ edges $e_{1}, e_{2}, \ldots, e_{c(D)}$. Each vertex corresponds to a Seifert circle and each edge corresponds to a crossing. Two distinct vertices $v_{i}$ and $v_{j}$ are connected by $e_{k}$ if the two Seifert circles $S_{i}$ and $S_{j}$ (corresponding to $v_{i}$ and $v_{j}$, respectively) are joined by the crossing $c_{k}$ (corresponding to $e_{k}$).

\begin{exm}
For the virtual singular knot considered in Example~\ref{exm1} (see also Figure~\ref{Seifert circles}), its Seifert graph has two vertices and three edges, see Figure~\ref{Seifert graph}.
 
 \begin{figure}[!htb]
\centering
\includegraphics[scale=0.15]{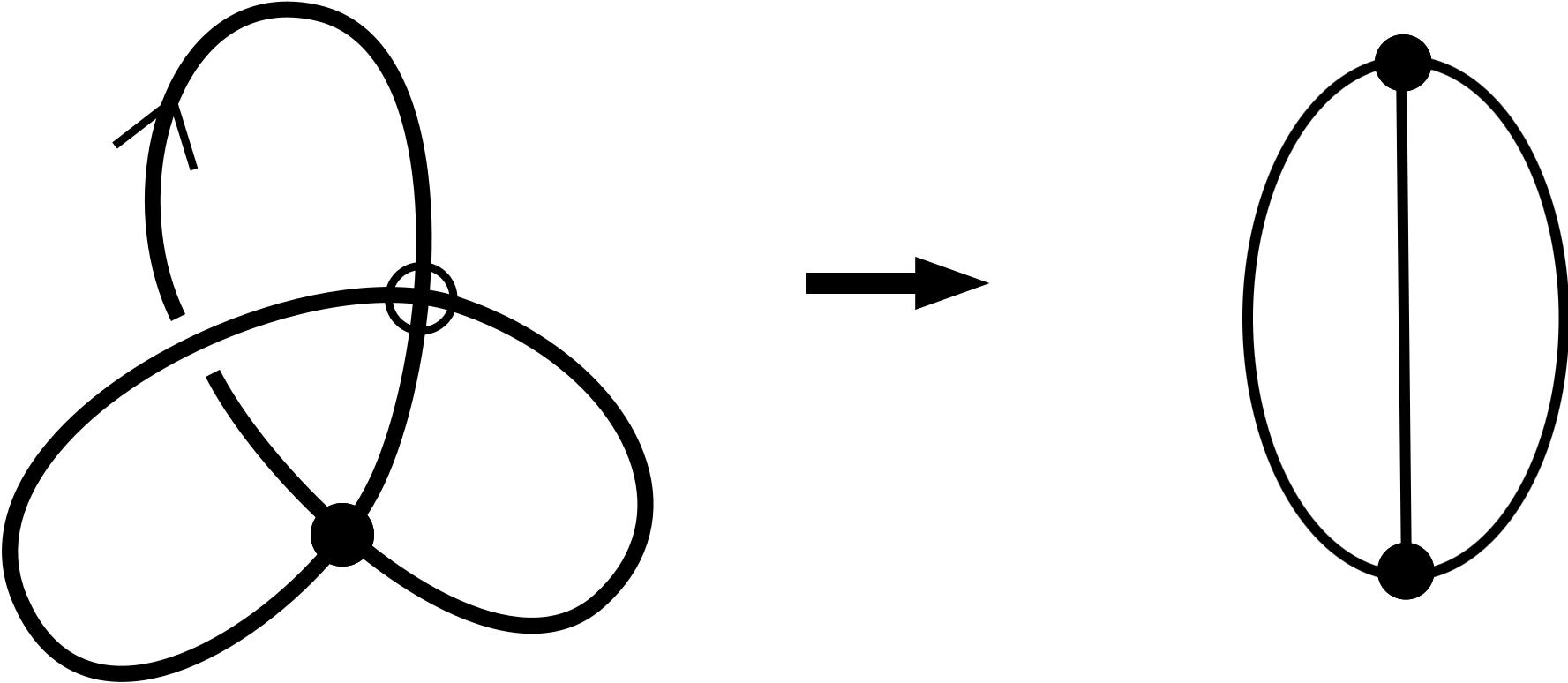}
\caption{The Seifert graph of a virtual singular knot.}
\label{Seifert graph}
\end{figure}
\end{exm}

Let $D$ be a non-classical oriented diagram, we denote by $cl(D)$ an oriented diagram, such that any non-classical crossing in $D$ is transformed into a classical crossing.
\begin{rem}
The choice of this transformation is not unique, since there are positive and negative real crossings.
\end{rem}

\begin{lem}\label{sg}
Let $D$ be a non-classical oriented diagram, then the Seifert graphs $\Gamma(D)$ and $\Gamma(cl(D))$ are equals.
\end{lem}
\begin{proof}
The proof is straightforward from the definition of Seifert graphs since $\Gamma(cl(D))$ is well-defined.
\end{proof}

\begin{prop}\label{sgeven}
The Seifert graph of any classical link diagram is bipartite.    
\end{prop}
\begin{proof}
Suppose that the Seifert graph of a given classical link diagram $D$ is not bipartite, i.e. it has a cycle with odd length. By the construction of the Seifert algorithm (see \cite[Theorem~5.1.1]{book1}), we have that $D$ is the boundary of a non-orientable surface, but this is a contradiction because all Seifert surfaces are orientable.    
\end{proof}

\begin{prop}
Let $D$ be a non-classical oriented diagram, then we have that the Seifert graph $\Gamma(D)$ is bipartite.   
\end{prop}
\begin{proof}
From Lemma~\ref{sg} we obtain $\Gamma(D) = \Gamma(cl(D))$ and then from Proposition~\ref{sgeven} it follows that the Seifert graph $\Gamma(cl(D))$ is bipartite.
\end{proof}

Therefore, we can conclude the main result of this section. 
 
\begin{thm}\label{sggnl}
The Seifert graph of generalized knot and link diagrams is planar and bipartite with $|E(\Gamma(D))| = c(D)$  and $|V(\Gamma(D))| = S(D)$.
\end{thm}

The following result is immediate from the previous theorem.

\begin{cor}
Let $L$ denote a virtual singular link or a doodle or a virtual doodle. 
The Seifert graph of a diagram $D$ of $L$ is planar and bipartite with $|E(\Gamma(D))| = c(D)$  and $|V(\Gamma(D))| = S(D)$.   
\end{cor}

\begin{rem}\label{rem:vs}
For virtual singular links there are some forbidden moves, see \cite[Figure~4]{article2}. 
Therefore, if we allow one or more forbidden moves in the collection of virtual singular link diagrams, we obtain other generalized links for which Theorem~\ref{sggnl} holds.
For example, the closure of generalized braids obtained in some quotients of virtual singular braid groups, see Definition~20 and Figure~7 of \cite{O}.
\end{rem}

We define the index of a diagram of a generalized link as the index of its Seifert graph as follows.

\begin{defn}\label{def:ind}
For an oriented diagram D of a oriented generalized link L, we define $$ind(D) = ind(\Gamma(D)).$$
\end{defn}

\begin{thm}\label{thm1}
If D is a diagram of a generalized link L with no nugatory crossings, then $$tc(D) \geq 2(S(D) - \textrm{ind}(D) - 1),$$ where $tc(D)$ is the total crossing number of $D$.
\end{thm}

\begin{proof}

From Definition~\ref{def:ind}, Theorems~\ref{2.4} and \ref{sggnl} we obtain that $c(D) \geq 2(S(D) - ind(D) - 1).$ Since $D$ is a diagram with no nugatory crossings, it follows that $c(D) = tc(D)$. Therefore, we have the desired result.
\end{proof}

We note that Theorem~\ref{thm1} gives the generalized version of the inequality established for the classical case in \cite[Theorem~3.7A]{article1} and for the virtual case in \cite[Theorem~2.6]{article3}.


\section{An inequality involving total crossing number and generalized braid index}\label{sec4}

It is important to emphasize that, in this section, we are considering generalized links, as links that admit an Alexander-type theorem and they admit Reidemeister-type moves $GR_{1}$, $GR_{2}$, $GR_{3}$ and $ GR_{4}$.

\begin{rem}\label{refinament}
   Note that, as a refinement of the Alexander-type theorem, we have that any diagram $D$ of a generalized link $L$, with $S(D)$ Seifert circles, can be deformed as the closure of a generalized braid with $S(D)$ strings.
\end{rem}

\begin{lem}\label{lem2}
  Let $L$ be a generalized link and $D$ its diagram. If $ind(D)$ is compatible with generalized crossings that admit Reidemeister-type moves $GR_{1}$, $GR_{2}$, $GR_{3}$ and $ GR_{4}$, then we have $$gb(L) \leq S(D) - ind(D),$$ where $gb(L)$ is the generalized braid index of $L$.  
\end{lem}
\begin{proof}
 Write $D = D_{1} \ast \ldots \ast D_{m}$ as a $\ast$-product of special diagrams $D_{i}$. Then $\Gamma(D) = \Gamma(D_{1}) \ast \ldots \ast \Gamma(D_{m})$. Since $\Gamma(D)$ is a planar bipartite graph, we have $\textrm{ind}(D) = \sum^{m}_{i=1} \textrm{ind}(D_{i})$ by Theorem~\ref{thmind}. First, we consider $\Gamma(D_{1})$ and if $\textrm{ind}(D_{1})=0$, we have nothing to do for $D_{1}$. Suppose $\textrm{ind}(D_{1})=k>0$. 
 Then there exists a singular virtual edge $e$ and a vertex $v$, one of the end points of $e$, such that $\textrm{ind}(\Gamma(D_{1})/\textrm{star}(v)) = k-1$. We denote the other end points of $e$ by $v_{0}$. In the following, we identify a vertex and its corresponding Seifert circle when there is no confusion. The edge $e$ corresponds to a generalized crossing $c$ of $D_{1}$, which consists of two short paths. We deform one of the short paths, say $u$, of $c$ along a long path $l$ by generalized Reidemeister-type moves, where $l$ is a path as depicted by the dotted curve in Figure~\ref{Figura 1 da demonstração}.

\begin{figure}[!htb]
\centering
\includegraphics[scale=1.0]{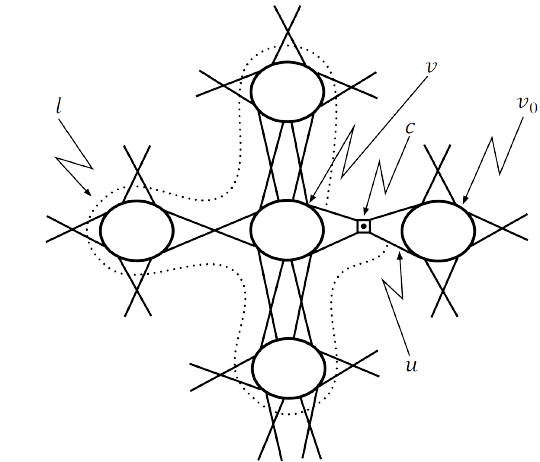}
\caption{A long path l, the Seifert circles, and the crossing near v.}
\label{Figura 1 da demonstração}
\end{figure}

We first deform $u$ by using $GR_{2}$ moves along $l$, but ``outside of $l$", as depicted in Figure~\ref{Figura 2 da demonstração}.

\begin{figure}[!htb]
\centering
\includegraphics[scale=1.0]{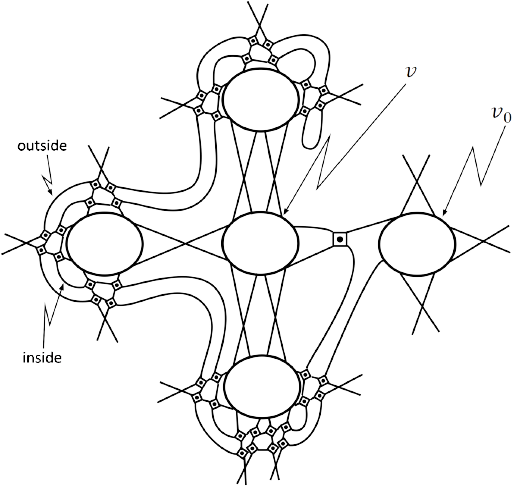}
\caption{A deformation of u by many $GR_{2}$ moves along l.}
\label{Figura 2 da demonstração}
\end{figure}

Then, we deform $u$ by using $GR_{2}, GR_{3}$ and $GR_{4}$ moves along $l$ as in Figure~\ref{Figura 3 da demonstração}, where we deform only the ``inside part".

\begin{figure}[!htb]
\centering
\includegraphics[scale=1.0]{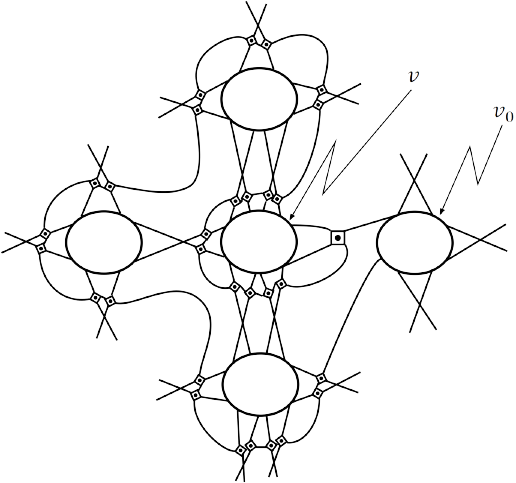}
\caption{A deformation of the ``inside part" by $GR_{2}, GR_{3}$ and $GR_{4}$ moves along l.}
\label{Figura 3 da demonstração}
\end{figure}

If $u$ meets the same type of crossing, then we use a $GR_{3}$ move, and if $u$ meets another type of crossing, then we use a $GR_{4}$ move. 
Next, we deform $u$ using $GR_{2}, GR_{3}$ and $GR_{4}$ moves as depicted in Figure~\ref{Figura 4 da demonstração}. 
Moreover, we deform $u$ by using a $GR_{2}$ move as depicted in Figure~\ref{Figura 5 da demonstração}. 

\begin{figure}[!htb]
\centering
\includegraphics[scale=1.0]{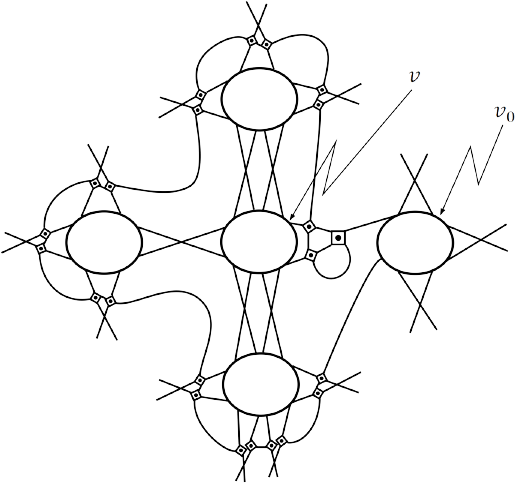}
\caption{A further deformation of the ``inside part" by $GR_{2}, GR_{3}$ and $GR_{4}$ moves.}
\label{Figura 4 da demonstração}
\end{figure}

\begin{figure}[!htb]
\centering
\includegraphics[scale=1.0]{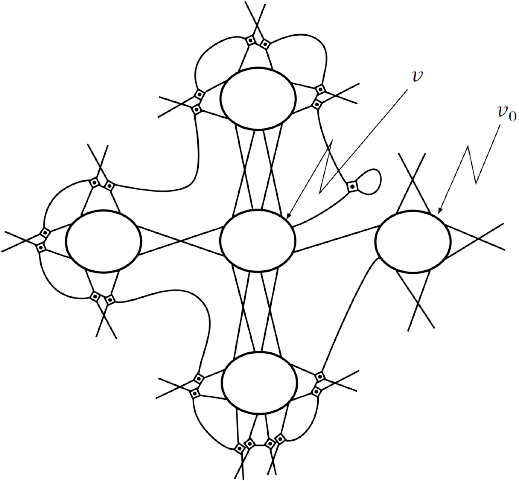}
\caption{Eliminating c by a $GR_{2}$ move.}
\label{Figura 5 da demonstração}
\end{figure}

Finally, we deform $u$ by using a $GR_{1}$ move as depicted in Figure~\ref{Figura 6 da demonstração}. 
In this way, we get a new diagram $D'_{1}$. Since the two Seifert circles represented by $v$ and $v_{0}$ are amalgamated to one circle $v_{1}$, we have $S(D'_{1}) = S(D_{1}) - 1$.
Now we see that $\Gamma(D'_{1})$ is the one point union of $\Gamma(D_{1})/\textrm{star}(v)$ and some multiple edge graph $K$, where $K$ contains $\textrm{star}(v)-e$ as a subgraph and $\textrm{ind}(\Gamma(D'_{1})) = k-1$.

\begin{figure}[!htb]
\centering
 \includegraphics[scale=1.0]{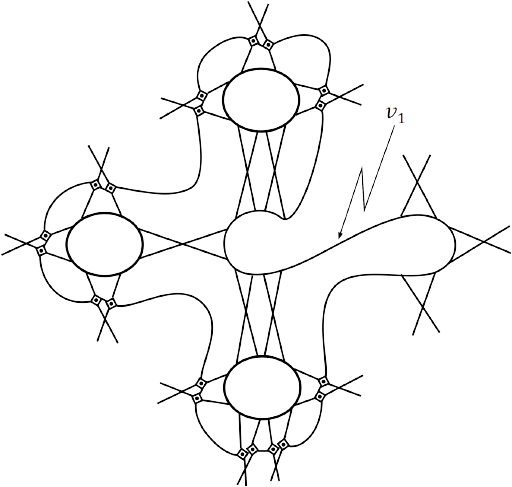}
\caption{Diagram having a new Seifert circle.}
\label{Figura 6 da demonstração}
\end{figure}

We can repeat the same argument $k$ times so that finally $\Gamma(D_{1})$ is reduced to the block sum of $\Gamma(D^{k}_{1})$ and $k$ multiple edges graphs $K, K^{1}, \ldots, K^{k-1}$, where $\textrm{ind}(\Gamma(D^{k}_{1})) = 0$.
Apply the same argument to each $\Gamma(D_{i})$ and eventually $\Gamma(D)$ is reduced to the block sum of $\Gamma(D^{k_{i}}_{i}), i=1,2,\ldots, m$, where $\textrm{ind}(\Gamma(D^{k_{i}}_{i})) = 0$ and the multiple edge graphs $K_{i}, K_{i}^{1}, \ldots, K_{i}^{k_{i}-1}$, $i=1,2,\ldots, m$. The final generalized link diagram $\widehat{D}$ corresponding to this graph has $S(\widehat{D}) = S(D) - \sum^{m}_{i=1} \textrm{ind}(D_{i})$. By Remark~\ref{refinament}, we have $gb(L) \leq S(\widehat{D})$. This complete the proof. 
      
\end{proof}

 Let $D$ be a diagram of a generalized link $L$ with no nugatory crossings, we define the {\it total crossing number} of $L$ as being the total crossing number of $D$. 
 Now, we state the main result of this paper. 
 
\begin{thm}\label{thm2}
   Let $D$ be a diagram of a generalized link $L$ with no nugatory crossings. If $\textrm{ind}(D)$ is compatible with generalized crossings which admit Reidemeister-type moves $GR_{1}$, $GR_{2}$, $GR_{3}$ and $ GR_{4}$, then we have $$ tc(L) \geq 2(gb(L) - 1), $$ where $tc(L)$ is the total crossing number of $L$ and $gb(L)$ is the generalized braid index of $L$.
\end{thm}

\begin{proof}
This is an immediate consequence of Theorem~\ref{thm1} and Lemma~\ref{lem2}, since $tc(D)=tc(L)$.

\end{proof}

We note that Theorem~\ref{thm2} gives a generalized version of the same kind of inequality established for the classical case in \cite[Theorem~3.8]{article1} and for the virtual case in \cite[Theorem~3.1]{article3}. 
For classical links, it is not necessary restrictions for the index $\textrm{ind}(D)$ since we consider only classical crossings, where we have $GR_{1},$ $GR_{2}$ and $GR_{3}$ type-moves (i.e. the classical Reidemeister moves). For virtual links, in \cite{article3} the author considered $\textrm{ind}(D) = \textrm{ind}_{0}(D)$ since only the crossings which admits $GR_{1},$ $GR_{2},$ $GR_{3}$ and $GR_{4}$ type-moves are the virtual ones, in this case they are the $V_{1}$, $V_{2}$, $V_{3}$ and $V_{4}$ moves. 
In a similar way, we have the following proposition.

\begin{prop}\label{cor:vsl}
  Let $D$ be a diagram of a virtual singular link $L$ with no nugatory crossings. If $ind(D) = ind_{0}(D)$, then we have $$ tc(L) \geq 2(vsb(L) - 1), $$ where $tc(L)$ is the total crossing number of $L$ and $vsb(L)$ is the virtual singular braid index of $L$. 
\end{prop}

\begin{proof}

Analogously to the proof of Theorem~\ref{thm2}, it follows from Theorem~\ref{thm1} and Lemma~\ref{lem2}, where we use Corollary~\ref{cor:ind} in place of Theorem~\ref{thmind}, since $\textrm{ind}(D) = \textrm{ind}_{0}(D)$. The Reidemeister-type moves $GR_{1},$ $GR_{2},$ $GR_{3}$ used are $V_{1},$ $V_{2},$ $V_{3},$ respectively. Moreover $V_{4}$ and $S_{3}$ are the $GR_{4}$ type-moves required. 
\end{proof}

\begin{rem}

Similarly to the claim of Remark~\ref{rem:vs}, we may extend Proposition~\ref{cor:vsl} to some other generalized links.

\end{rem}


\end{document}